\documentclass{article}
\usepackage{graphicx}
\usepackage{amsmath}
\usepackage{amsfonts}
\usepackage{amsthm}
\usepackage{amssymb}
\usepackage[hidelinks]{hyperref}
\usepackage[affil-it]{authblk}
\usepackage{hyperref}
\newcommand{\tarc}{\mbox{\large$\frown$}}
\newcommand{\arc}[1]{\stackrel{\tarc}{#1}}

\newtheorem{theorem}{Theorem}[section]
\newtheorem{lemma}[theorem]{Lemma}

\newtheorem{prop}[theorem]{Property}

\newtheorem{defn}[theorem]{Definition} 
\newtheorem{corollary}[theorem]{Corollary} 

\begin{document}

\title{Modified contact simplex iteration}
\author{Sergei Drozdov\footnote{email: \href{mailto:smdrozdov@gmail.com}{smdrozdov@gmail.com} }}
\maketitle
\begin{abstract}
We study the iterations of the procedure of taking the contact simplex.
We define the concept of the root of the simplex, which is a homothety image of contact simplex with a special coefficient greater than 1.
The article shows that once we iterate the root of a given simplex, the circumcenter sequence of these simplices has two partial limits.
\end{abstract}

\section{Introduction}
Consider triangle $S,$ and its inscribed circle $C.$ Then the triangle formed by the points of intersection of $S$ and $C$ is called the \textbf{contact triangle} (also known as Gergonne triangle or intouch triangle) of $S$. If we iterate the procedure of taking the contact triangle, the sequence obtained will converge to the equilateral triangle (if the proper normalisation is applied). This follows, for example, from the fact that on each step we switch from the triangle with the angles
$\left(\frac{\pi}{3}+\alpha,\frac{\pi}{3}+\beta,\frac{\pi}{3}+\gamma\right)$ to the triangle with the angles
$\left(\frac{\pi}{3}-\frac{\alpha}{2},\frac{\pi}{3}-\frac{\beta}{2},\frac{\pi}{3}-\frac{\gamma}{2}\right)$.
\begin{figure}[htp]
    \centering
    \includegraphics[width=12cm]{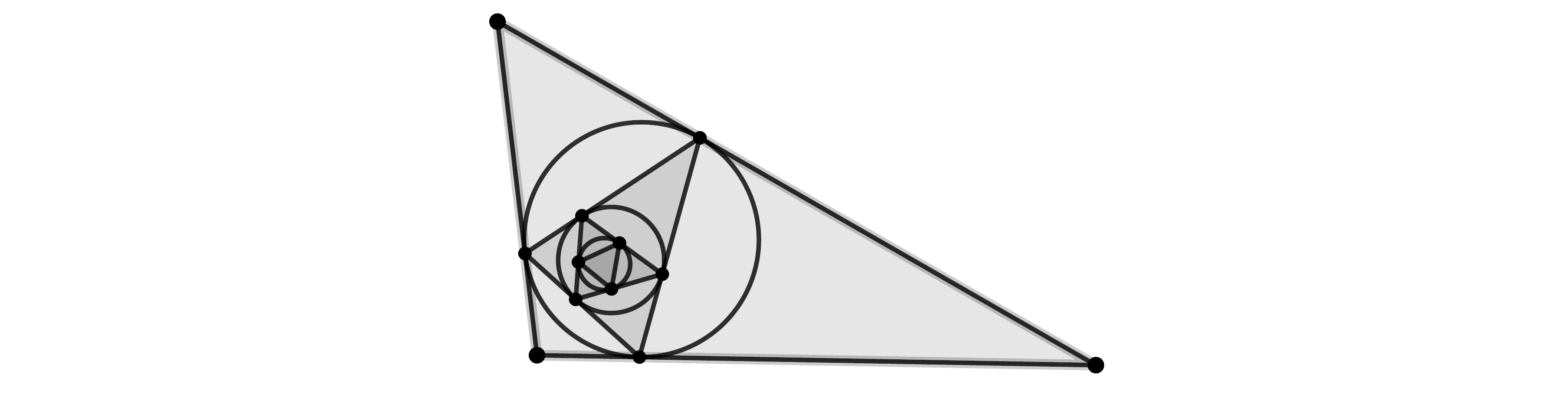}
    \caption{Contact triangle iteration.}
    \label{fig:root-definition}
\end{figure}

The same procedure for the euclidian simplex of the higher dimension is discussed in \cite{Fedja}, and for the tetrahedrons in \cite{Krizek}. As shown in these sources, the sequence of simplices does not have a unique limit, but instead, the iterations of the odd and even orders have two limits, that are, generally speaking, different.

We modify the procedure slightly, and instead of normalizing contact simplices, let them grow (in a special way). We prove the technical Lemma \ref{lemma_container} and afterward the Theorem \ref{main_theorem}, which shows that the circumcenter sequence of these simplices has two partial limits.
 
\section{Root of the simplex definition}

\begin{defn} \textbf{Root.} Let $S \subset \mathbb{R}^n$ be a nondegenerate simplex with $n+1$ vertex. Denote $S$ by $A_1 A_2 \dots A_{n+1}.$  Let $I$ and $r$ be the center and the radius of inscribed sphere of $S$. Let the full-dimension face opposite to $A_i$ be $\Delta_i$. $\forall i$ the inscribed sphere touches $\Delta_i$ at some point $B_i.$
Let the radius of circumscribed sphere of $S$ be $R,$ and consider space transformation $K$, which is the composition of the inversion with radius $R$ and center $I$ and a central symmetry, centered at point $I.$
Then, the simplex $K(B_1)K(B_2) \dots K(B_{n+1})$ is called the \textbf{root} of the simplex $S.$ See Figure \ref{fig:root-definition}.
\label{def_root}
\end{defn}
Denote it as $Root(S)$.

\begin{figure}[htp]
    \centering
    \includegraphics[width=12cm]{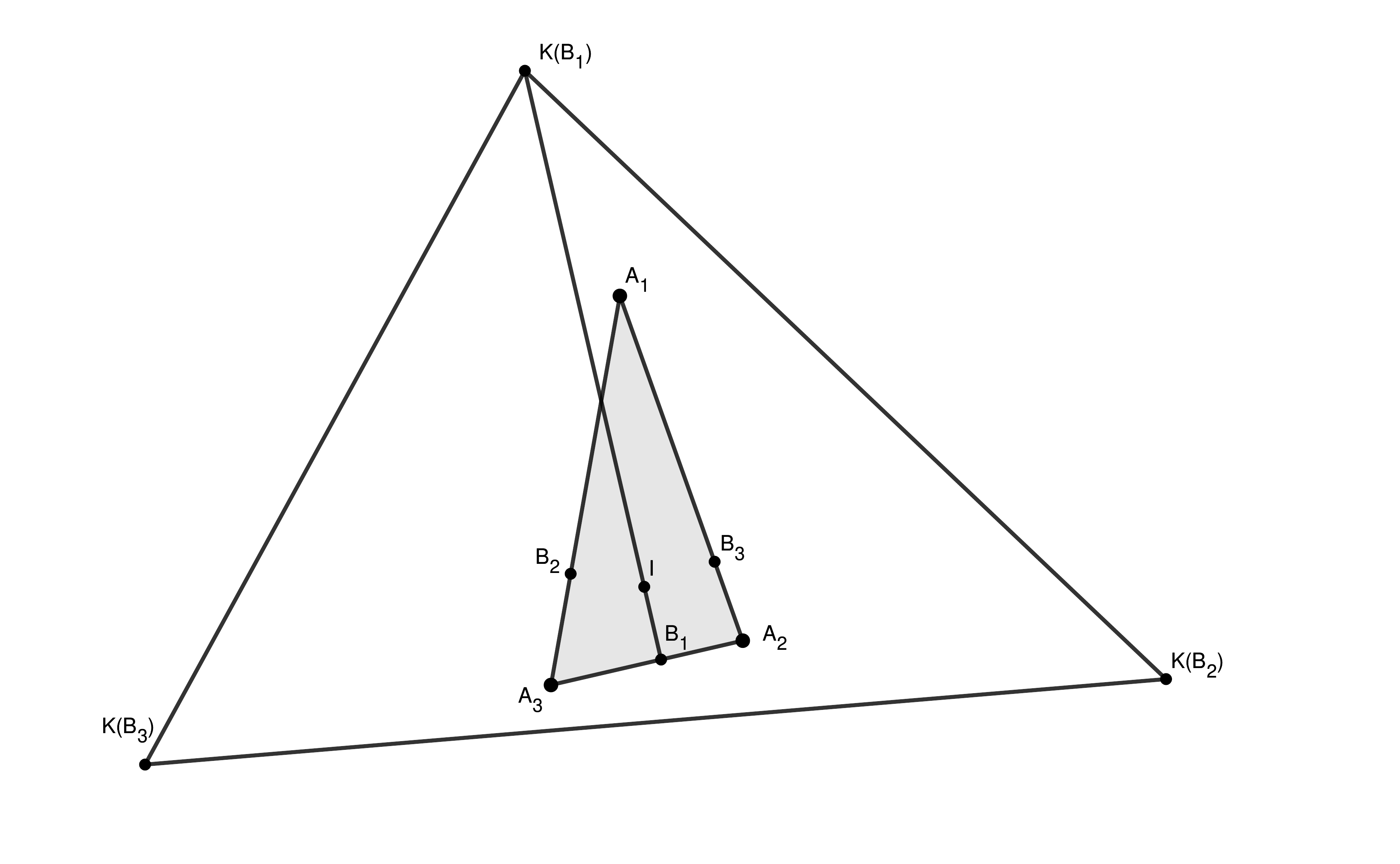}
    \caption{$Root(\triangle A_1A_2A_3)=\triangle K(B_1) K(B_2) K(B_3)$.}
    \label{fig:root-definition}
\end{figure}

\begin{prop}
Let $S \subset \mathbb{R}^n$ be a nondegenerate simplex and $T = Root(S).$ Then $T$ is homothetic to the contact simplex of $S.$  And also, $I$ and $\frac{R^2}{r}$ are the center and the radius of the circumscribed sphere of $T.$
\label{prop_radii}
\end{prop}
\begin{proof}
In the notation of the Definition \ref{def_root}  the center and the radius of the circumscribed sphere of the simplex 
$B_1\dots B_{n+1}$ are $I$ and $r.$
Simplex $B_1\dots B_{n+1}$ is the contact simplex of $A_1\dots A_{n+1}.$
The pointwise inversion image of $B_1\dots B_{n+1}$ is homothetic to $B_1\dots B_{n+1}$ with the coefficient $\frac{R^2}{r^2}$.
If the apply central symmetry, centered at point $I,$ the resulting image is homothetic to $B_1\dots B_{n+1}$ with the coefficient $-\frac{R^2}{r^2}$.
Thus, simplex $T$ has the circumscribed sphere with the same center $I$ and with the radius equal $r \cdot \frac{R^2}{r^2}=\frac{R^2}{r}.$
\end{proof}

We will further denote simplex $T$ by $C_1 \dots C_{n+1}.$
\begin{prop}
In the notation of the Definition \ref{def_root}, $\forall i \neq j : 1 \leqslant  i, j \leqslant n + 1$
\begin{equation}
\overrightarrow{IC_i} \cdot \overrightarrow{IA_j} = -R^2
\end{equation}
\label{prop_gmt}
\end{prop}
\begin{proof}
\begin{equation}
\overrightarrow{IC_i} \cdot \overrightarrow{IA_j} =
-\frac{R^2}{r^2} \cdot  \overrightarrow{IB_i} \cdot \overrightarrow{IA_j} ,
\end{equation}
and since $I$ is the center of inscribed sphere of $S$, and so $IB_i \perp B_i A_j,$ we can continue:
\begin{equation}
-\frac{R^2}{r^2} \cdot  \overrightarrow{IB_i} \cdot \overrightarrow{IA_j}=
-\frac{R^2}{r^2} \cdot  IB_i ^ 2 = 
-\frac{R^2}{r^2} \cdot r^2 = -R^2
\end{equation}
\end{proof}

\begin{prop}
$I \in Int (T)$
\end{prop}
\begin{proof}
As it is shown in \cite{Krizek2}, in Theorem 3, the circumcenter of simplex $B_1 \dots B_{n+1}$ is an interior point of $B_1 \dots B_{n+1}.$
So, since $T$ is homothetic to $B_1 \dots B_{n+1},$ the circumcenter of $T$ is an interior point of $T.$
\end{proof}

\section{Root of the simplex contains the simplex circumscribed sphere}
\begin{lemma}
Let $S \subset \mathbb{R}^n$ be the simplex, and $T=Root(S).$ Then $T$ contains the circumscribed sphere of $S.$
\label{lemma_container}
\end{lemma}
\begin{proof}
\begin{figure}[htp]
    \centering
    \includegraphics[width=12cm]{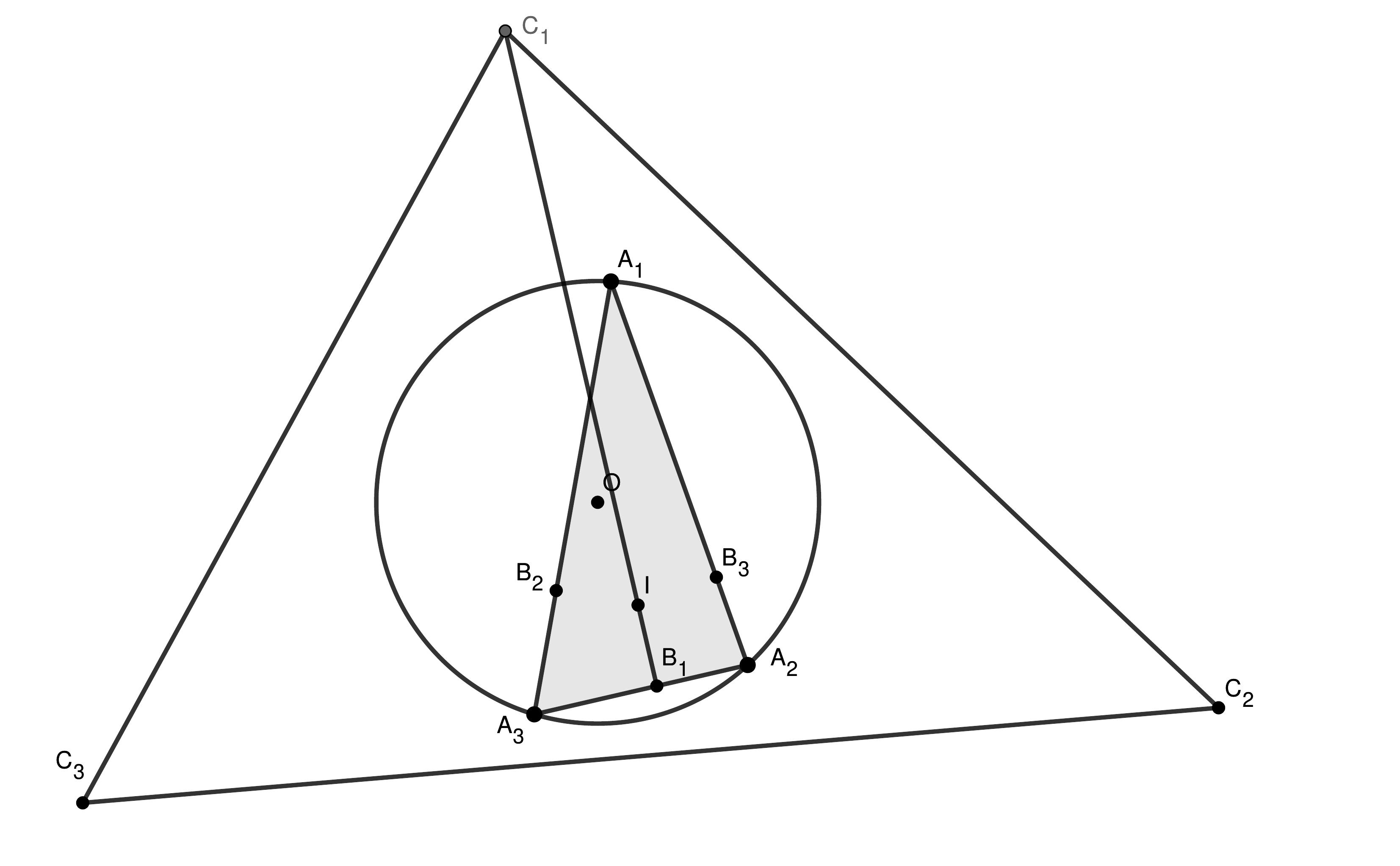}
    \caption{Statement of the Lemma \ref{lemma_container} for the triangle ($n=2$): $\triangle C_1C_2C_3$ contains the circumscribed circle of $\triangle A_1 A_2 A_3.$}
    \label{fig:container-lemma}
\end{figure}
(Please notice that here we meet the circumscribed sphere of $S$ for the first time. See Figure \ref{fig:container-lemma}.)

Let again $S = A_1 \dots A_{n+1},$ and denote $T$ by $C_1 \dots C_{n+1}$. Let $O$ and $I$ be the centers of the
circumscribed and inscribed spheres of $S.$ Denote the hyperplane containing the full-dimension face of $T$, that is opposite to $C_i$ by 
$\Gamma_i$. So

\begin{equation}
\Gamma_i := Lin\left<C_1 \dots C_{i-1}C_{i+1} \dots C_{n+1} \right>.
\end{equation}
From the Property \ref{prop_gmt}
$\forall j \neq i \quad \overrightarrow{IC_j} \cdot \overrightarrow{IA_i} = -R^2,$ where $R$ is the radius of circumscribed sphere of $S,$ so
\begin{equation}
\Gamma_i = \left\{ X: \overrightarrow{IX} \cdot \overrightarrow{IA_i} = -R^2 \right\}.
\end{equation}
This leads to
\begin{equation}
\Gamma_i \perp \overrightarrow{IA_i},
\label{IA_perp}
\end{equation}
and if we denote the intersection of $(IA_i)$ and $\Gamma_i$ by $H,$ then

\begin{equation}
 \overrightarrow{IH} \cdot \overrightarrow{IA_i} = -R^2,
\end{equation}

or in other words,

\begin{equation}
\left\{
\begin{aligned}
I \in [HA_i]\\
|IH| \cdot |IA_i| = R^2.\\ \end{aligned}
\right.
\label{prod_property}
\end{equation}

See Figure \ref{fig:projection}.
\begin{figure}[htp]
    \centering
    \includegraphics[width=12cm]{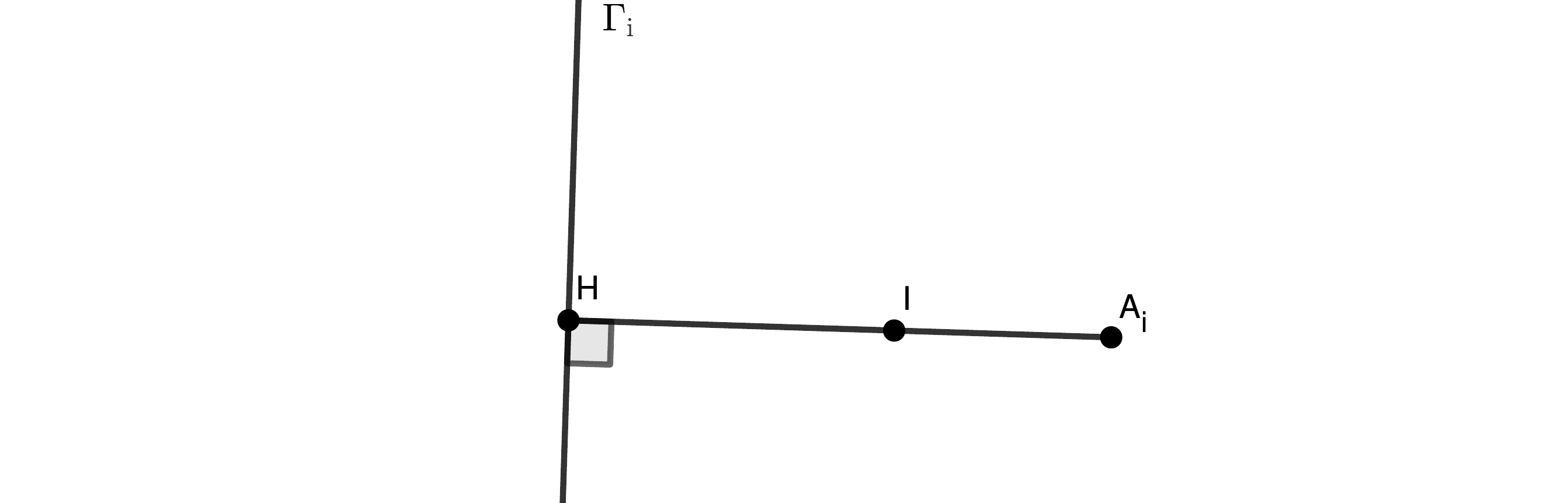}
    \caption{Point $H$ definition.}
    \label{fig:projection}
\end{figure}

Now let's consider the point $O,$ and let the projection of $O$ on $\Gamma_i$ be $G.$
See Figure \ref{fig:projection2}.
\begin{figure}[htp]
    \centering
    \includegraphics[width=12cm]{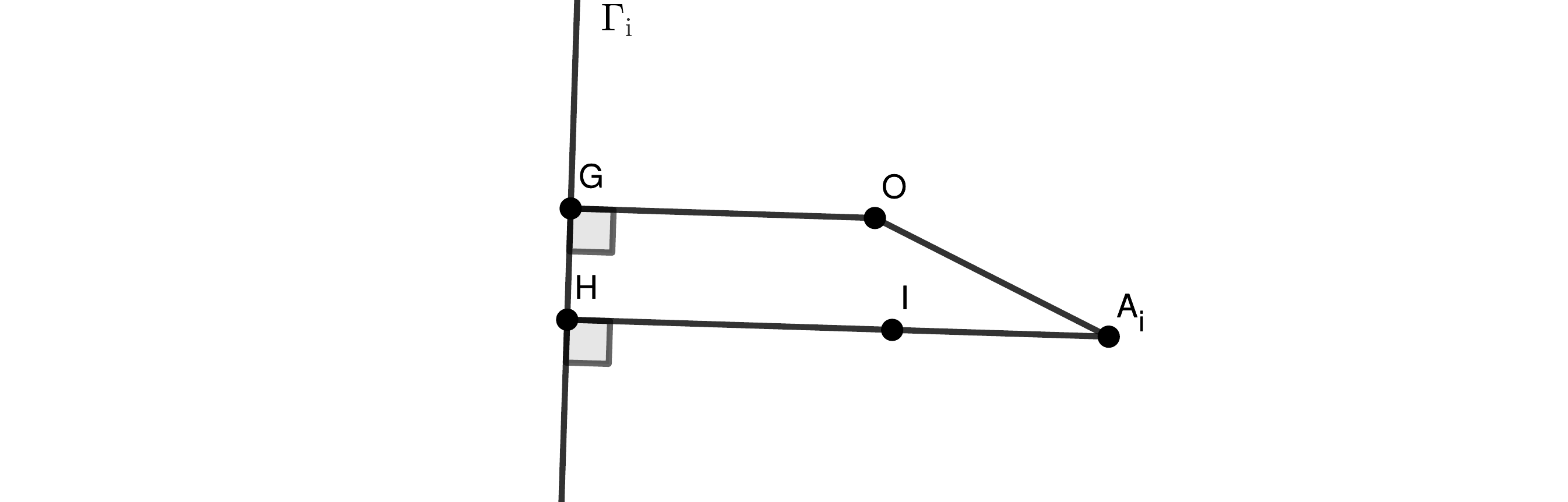}
    \caption{Point $G$ definition}
    \label{fig:projection2}
\end{figure}

\begin{equation}
|HA_i| = |HI| + |IA_i| \geqslant 2 \sqrt{|HI| \cdot |IA_i|} = 2 \sqrt{R^2} = 2R.
\label{2R}
\end{equation}
Since $|OA_i| = R,$ we can say:
\begin{equation}
|A_iH| \leqslant |A_iO| + |OG| = R + |OG|
\end{equation}
\begin{equation}
|OG| \geqslant |A_iH| - R \geqslant 2R - R = R.
\label{lower_bound_og}
\end{equation}

It means that the open ball circumscribed around $S$ does not cross the plane $\Gamma_i.$
The same holds for all $i = 1,..,n+1.$ 
So, the ball does not cross any of the faces of $T.$
The only part left is to show that it lies inside $T,$ and not outside.

The hyperplane $\Gamma_i$ cuts the space in two half-spaces.
Let $\Pi_i$ be the half-space containing the ball circumscribed around $S.$
The inscribed ball of $S$ is contained inside the circumscribed ball of $S,$ so it also lies in $\Pi_i.$
Then so does its center, point $I.$
So, $I \in \Pi_i.$ 
From the Property \ref{prop_gmt}, $I \in Int(T),$ so the half-space $\Pi_i$ contains not only $I,$ but also the
whole simplex $T.$

So, $T = \cap_{i =1}^{n+1} \Pi_i,$ and since $\forall i $ the half-space $\Pi_i$ contains the circumscribed ball of $S,$ this ball is contained in $T.$
\end{proof}

\begin{corollary}
Let $S \subset \mathbb{R}^n$ be a simplex with $n+1$ vertices. $T = Root(S).$ Let $(r_1, R_1)$ be the radii
of inscribed and circumscribed spheres of $S,$ and $(r_2, R_2)$ be the radii of inscribed and circumscribed spheres of $T.$ Then
\begin{equation}
\frac{r_2}{R_2} \geqslant \frac{r_1}{R_1}
\end{equation}
\end{corollary}
\begin{proof}
From Property \ref{prop_radii}, $R_2 = \frac{R_1^2}{r_1},$ from Lemma \ref{lemma_container} the circumscribed sphere of $S$ lies inside $T.$ Any sphere inside $T$ has radius no larger than the radius of inscribed sphere of $T.$ So
\begin{equation}
R_1 \leqslant r_2,
\label{container_inequality}
\end{equation}

and
\begin{equation}
\frac{r_2}{R_2} = \frac{r_2}{\frac{R_1^2}{r_1}}= \frac{r_1 r_2}{R_1 ^ 2} \geqslant \frac{r_1 R_1}{R_1^2} =
\frac{r_1}{R_1}.
\end{equation}

\end{proof}
\begin{corollary}
Let $S_1 \in \mathbb{R}^n$ be a simplex with $n+1$ vertices. For every $k \geqslant 1 $ define $ S_{k+1}:= Root(S_k),$ and
let $(r_k, R_k)$ be the radii of inscribed and circumscribed spheres of $S_k.$ Then $\frac{r_k}{R_k}$ converges as
$k \to \infty.$ Also, denote $\lim \limits_{k\to \infty} \frac{r_k}{R_k}  = \rho,$ then $\forall k \enskip \frac{r_k}{R_k}  \leqslant \rho$
\label{corollary_limiting_radius}
\end{corollary}
\begin{proof}
$\frac{r_k}{R_k}$ is non-decreasing, and $\forall k \enskip r_k < R_k.$ So $\frac{r_k}{R_k}$ converges.
\end{proof}

\begin{prop}
In the notation as above, $\lim\limits_{k \to \infty} \frac{r_k}{R_k} \leqslant \frac{1}{n-1}.$
\label{rho_bound}
\end{prop}
\begin{proof}
Since, for any $n+1-$vertex simplex in $\mathbb{R}^n,$ with the radii of inscribed and circumscribed spheres $r$ and $R$ holds $\frac{r}{R} \leqslant \frac{1}{n-1},$ then for all $k$ holds
$\frac{r_k}{R_k} \leqslant \frac{1}{n-1},$ so the limit of $\frac{r_k}{R_k}$ is smaller than or equals $\frac{1}{n-1}.$
\end{proof}

\begin{prop}
In the notation as above, if $n=2$ ($S_1$ is a triangle), then $\lim\limits_{k \to \infty} \frac{r_k}{R_k} = \frac{1}{2}.$
\end{prop}
\begin{proof}
If $n=2,$ the angles of $S_k$ converge to $\frac{\pi}{3},$ as $k \to \infty.$ In the neighborhood of the equilateral triangle, the factor of the radii of inscribed and circumscribed circles of a triangle is a continuous function of the angles. So this factor converges to the value of the factor in the equilateral triangle.
There it equals $\frac{1}{2}.$
\end{proof}

\section{Main result}
\begin{theorem}
\label{main_theorem}
Let $S_1 \in \mathbb{R}^n$ be a simplex with $n+1$ vertices. For every $k \geqslant 1$ define $S_{k+1}:= Root(S_k),$ and let $(I_k, O_k)$ be the centers of inscribed and circumscribed spheres of $S_k.$ Then the sequences 

\begin{equation}
	\left\{O_{2k+1},k\in \mathbb{N}\right\} \text{ and } \left\{O_{2k},k\in \mathbb{N}\right\}
\end{equation}
converge as $k \to \infty.$
\end{theorem}

\begin{proof}
\begin{figure}[htp]
    \centering
    \includegraphics[width=12cm]{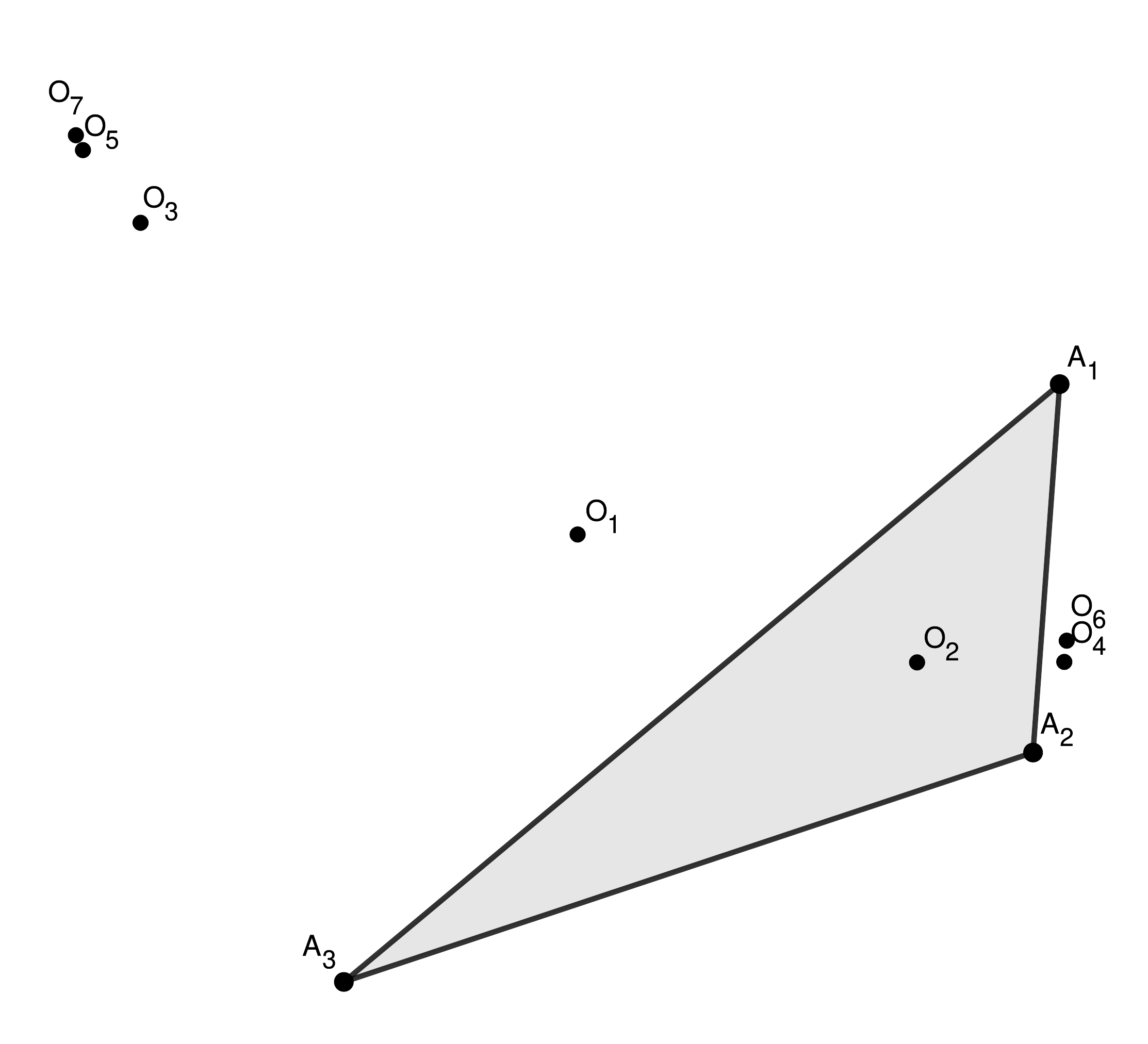}
    \caption{Statement of the Theorem \ref{main_theorem} for the triangle ($n=2$): given $S_1 = \triangle A_1A_2A_3,$ point $O_1$ being its circumcenter, and $O_k$ being the circumcenter of $Root^{k-1}(S_1), k > 1,$ sequences $O_2, O_4, O_6, \dots$ and
    $O_1, O_3, O_5,\dots$ both have limits in $\mathbb{R}^2.$}
    \label{fig:iterations}
\end{figure}
(See Figure \ref{fig:iterations}.)

From the Corollary \ref{corollary_limiting_radius} the sequence $\frac{r_k}{R_k}$ has its limit, let's call it $\rho.$ Then
$\forall \varepsilon >0 \enskip  \exists K: \forall k  \geqslant K$ we have 

\begin{equation}
	\rho \geqslant \frac{r_k}{R_k} \geqslant \rho(1 - \varepsilon).
	\label{lower_bound_radius}
\end{equation}

Assume further that $\varepsilon < \frac{1}{1000}$ for all $k > K_1,$ and consider only these $k.$

\textbf{1. Estimate on $|I_k O_k|$ from above.}

From Lemma \ref{lemma_container} the circumscribed sphere of $S_k$ lies inside $S_{k+1},$ and so does its center, point $O_k.$

So $O_k \in Int(S_{k+1}).$

Since $r_{k+1} \leqslant \rho R_{k+1},$ for any point $X \in Int(S_{k+1})$ the simplex $S_{k+1}$ has a full-dimension face $\Gamma,$ such that $dist(X,\Gamma) \leqslant \rho R_{k+1}.$

And so, for $X=O_k$ the simplex $S_{k+1}$ has a full-dimension face $\Gamma,$ such that 
\begin{equation}
dist(O_k,\Gamma) \leqslant \rho R_{k+1}.
\end{equation}
From the Property \ref{prop_radii} we can get that:
\begin{equation}
\rho R_{k+1}=\rho\frac{R_k^2}{r_k}.
\end{equation}
Using (\ref{lower_bound_radius}) we can continue:

\begin{equation}
dist(O_k,\Gamma) \leqslant \rho\frac{R_k^2}{r_k}  \leqslant \frac{\rho R_k}{\rho (1 - \varepsilon)}=
 \frac{R_k}{1 - \varepsilon}.
 \label{upper-bound-distance}
\end{equation}

Denote the vertex of $S_k$ corresponding to $\Gamma$ by $A$ ($\Gamma$ is a face of $S_{k+1}$, not $S_k$, so it is not "opposite" in the regular sense). 

Let the projection of $O_k$ on $\Gamma$ be $G.$
Let the intersection of $(AI_k)$ and $\Gamma$ be $H.$
From (\ref{IA_perp}) we know that $AH \perp \Gamma.$
See Figure \ref{fig:projection3}.
\begin{figure}[!htbp]
    \centering
    \includegraphics[width=12cm]{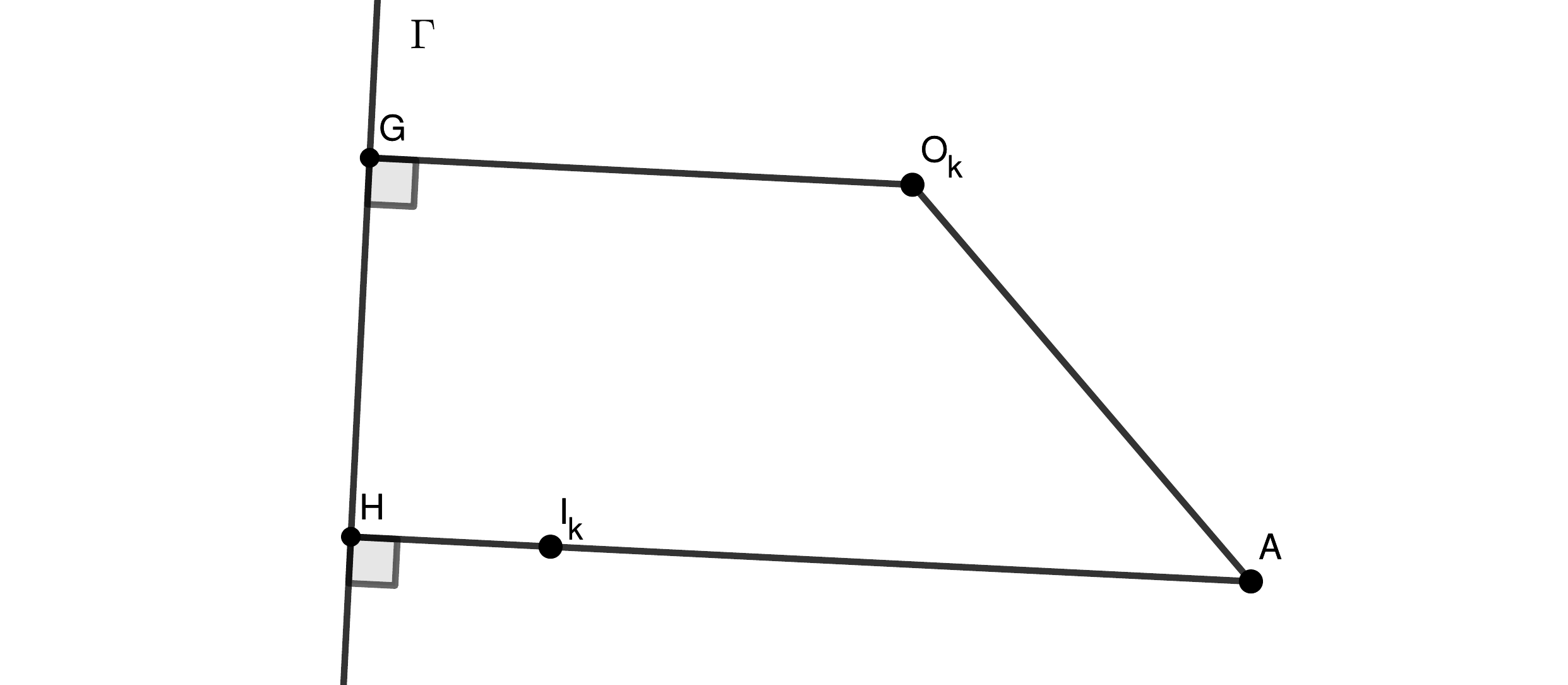}
    \caption{Points $G$ and $H$.}
    \label{fig:projection3}
\end{figure}
We will further write $R$ instead of $R_k.$ 
Remember that from (\ref{2R}) we have
\begin{equation}
	|AH| \geqslant 2R.
	\label{2R2}
\end{equation} 
Let the projection of $O_k$ on $AH$ be $X.$ See Figure \ref{fig:projection6}.
\begin{figure}[!htbp]
    \centering
    \includegraphics[width=12cm]{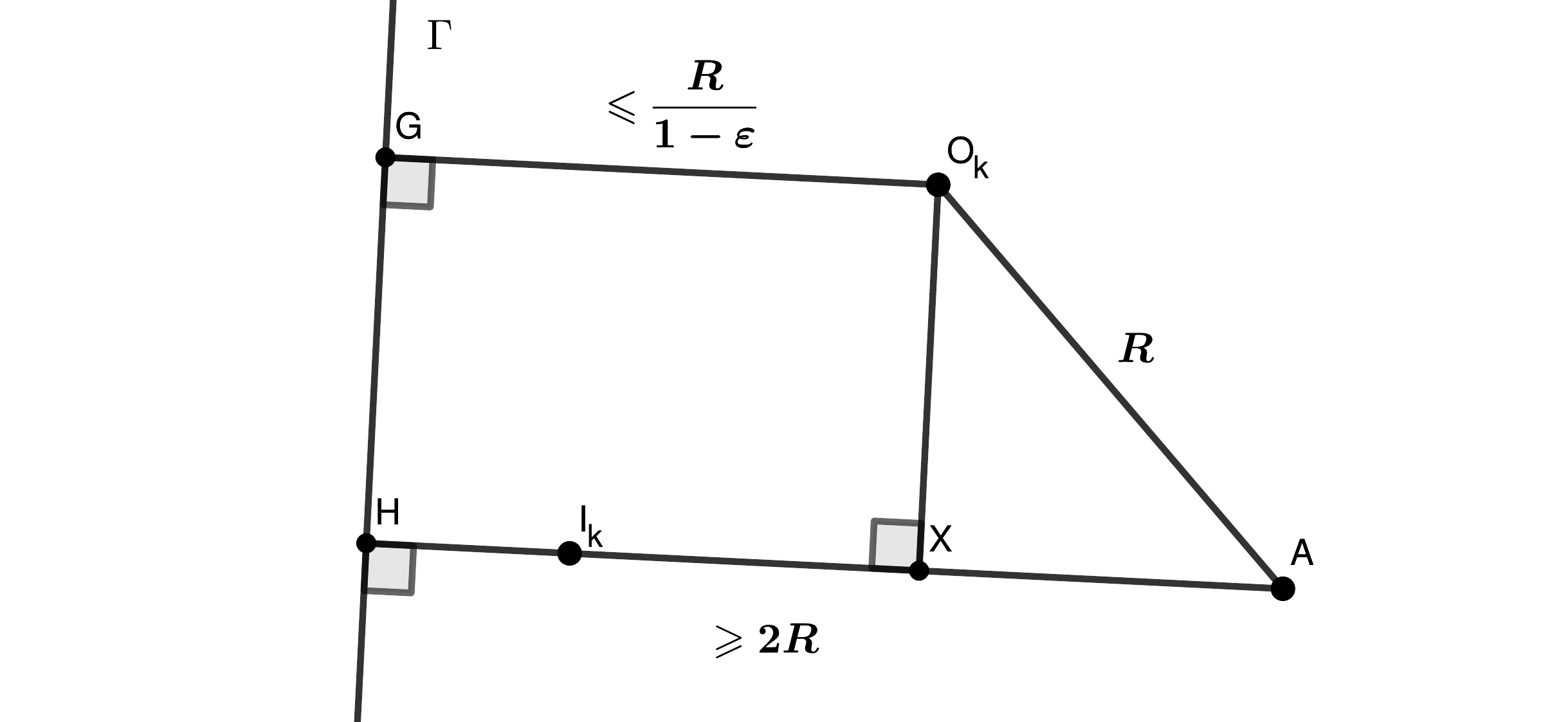}
    \caption{Point $X$ definition.}
    \label{fig:projection6}
\end{figure}

From (\ref{upper-bound-distance}) we have:
\begin{equation}
	|XH| =|O_k G|  =  dist(O_k,\Gamma) \leqslant  \frac{R}{1 - \varepsilon}.
	\label{XH_upper}
\end{equation}

For $\varepsilon < \frac{1}{2}$ we get two corollaries from the combination of (\ref{2R2}) and (\ref{XH_upper}):
first is that, since $O_k,X,A$ all lie in the same half-space with respect to $\Gamma,$ $X$ lies inside the interval $[AH],$ and second is:
\begin{equation}
	|AX| = |AH| - |HX| \geqslant 2R - |HX|  \geqslant 2R - \frac{R}{1 - \varepsilon} = R\left( 2 - \frac{1}{1-\varepsilon}\right).
\end{equation}
There is now a lower bound on $\frac{|AX|}{R}.$
More specifically, we know, that as $\varepsilon \to 0, \enskip \frac{|AX|}{R} \to 1.$
And also, $|AX|$ is a cathetus in the right triangle $\triangle XAO_k,$ with the hypothenuse $O_k A$ equal $R.$
This allows estimating from above $\angle HAO_k.$
 
Let $\alpha := \angle HAO_k,$ then first, $\alpha < \frac{\pi}{2}.$ Since $AX = R \cos \alpha:$
\begin{equation}
 R \cos \alpha \geqslant  R\left( 2 - \frac{1}{1-\varepsilon}\right).
 \label{rcos}
\end{equation}
From (\ref{rcos}) we have:
\begin{equation}
 \cos \alpha \geqslant   2 - \frac{1}{1-\varepsilon}
\end{equation}
\begin{equation}
\alpha \leqslant \arccos \left( 2 - \frac{1}{1-\varepsilon}\right)
\label{alpha_upper}
\end{equation}

(which means that $\alpha \to 0,$ as $\varepsilon \to 0$).

Let $O' \in [AH]$ be a point, such that $|O'A|=R.$ Such point exists since $|AH| \geqslant 2R.$
See Figure \ref{fig:projection4}.
\begin{figure}[htp]
    \centering
    \includegraphics[width=12cm]{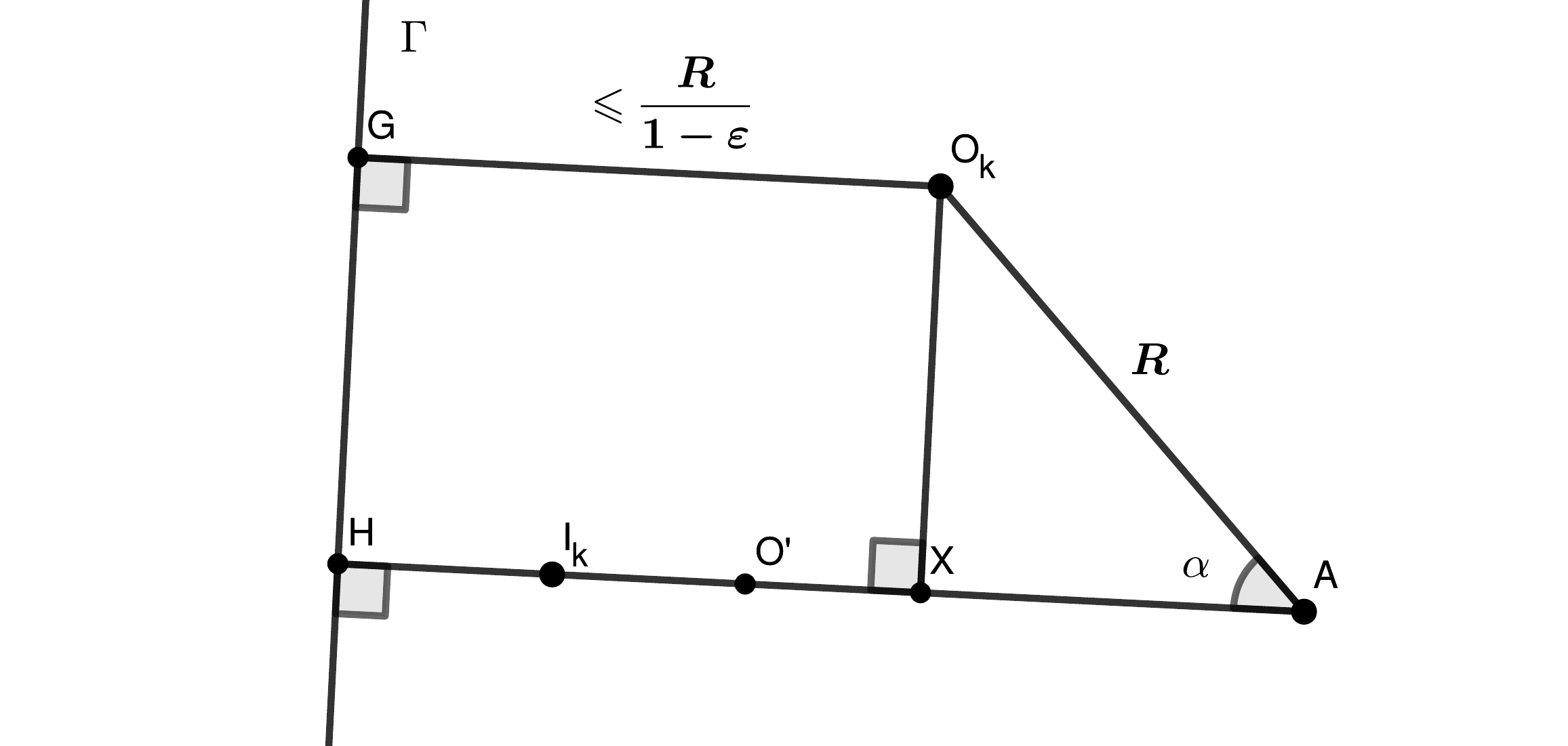}
    \caption{Point $O'$ definition.}
    \label{fig:projection4}
\end{figure}

$|O_k A| = |O'A| = R.$ The length of the interval $[O_k O' ]$ is no larger than the
length of the arc $\arc{O_k O'} $ in the circle with radius $R$ and center $A.$ Combining this with (\ref{alpha_upper}) we get:

\begin{equation}
|O_kO'|  \leqslant R\alpha \leqslant R \arccos \left( 2 - \frac{1}{1-\varepsilon}\right).
\label{eq_oodash}
\end{equation}

The chapter's goal is to estimate $|O_k I_k|$ from above, and we have an estimate on $|O'O_k|.$
Since $|O_k  I_k| \leqslant |O_k O'| + |O' I_k|,$ it is only left to estimate $|O'I_k|$ from above.

First, using that $\angle O_k X A = \frac{\pi}{2},$ and (\ref{XH_upper}) we estimate $|AH|:$

\begin{equation}
	|AH| = |AX| + |XH| \leqslant R + |XH| \leqslant R + \frac{R}{1 - \varepsilon} = R \left( 1 + \frac{1}{1 - \varepsilon}\right).
	\label{AH_upper}
\end{equation}

Now, (\ref{prod_property}) says:
\begin{equation}
	|AI_k| \cdot |I_kH| = R^2.
\end{equation}

Let's denote $|AI_k| - |AO'|=|AI_k| - R,$ the oriented length of $O'I_k$ by $z.$ Then

\begin{equation}
	|AH| = |AI_k| + |I_kH| = |AI_k| + \frac{R^2}{|AI_k|} = R + z + \frac{R^2}{R + z}.
	\label{AH_eq}
\end{equation}

Combining (\ref{AH_upper}) and (\ref{AH_eq}) we get:
\begin{equation}
	R + z + \frac{R^2}{R + z} \leqslant   R \left( 1 + \frac{1}{1 - \varepsilon} \right),
\end{equation}
and dividing both parts by $R,$ we write:

\begin{equation}
	1 + \frac{z}{R} + \frac{1}{1 + \frac{z}{R}} \leqslant   1 + \frac{1}{1 - \varepsilon},
\end{equation}

Denote $\frac{z}{R} $ by $y.$
So
\begin{equation}
	1 + y+ \frac{1}{1 + y} \leqslant   1 + \frac{1}{1 - \varepsilon}.
\end{equation}

As we assumed, that $k > K_1$ in the beginning, and so  $\varepsilon < \frac{1}{1000},$ we can use the Taylor expansion for $\varepsilon$, and write that
\begin{equation}
	 1 + \frac{1}{1 - \varepsilon} = 1 + 1 + \varepsilon + \varepsilon^ 2 + \dots \leqslant 2 + 2 \varepsilon,
\end{equation}
and so:
\begin{equation}
	1 + y+ \frac{1}{1 + y} \leqslant   2 + 2 \varepsilon.
\end{equation}
As the denominator $1+y$ is larger than 0 (it is equal to $\frac{|AI_k|}{R}$), we can multiply by $1+y$, and obtain:

\begin{equation}
	(1 + y)^2+ 1 \leqslant   (1+y)(2 + 2 \varepsilon),
\end{equation}
so
\begin{equation}
	y^2 + 2y + 2 \leqslant 2 + 2y + 2\varepsilon y + 2 \varepsilon,
\end{equation}
so
\begin{equation}
	y^2 -2\varepsilon y \leqslant  2 \varepsilon,
\end{equation}
so
\begin{equation}
	(y-\varepsilon)^2  \leqslant 2 \varepsilon + \varepsilon ^ 2,
\end{equation}
so
\begin{equation}
	|y| \leqslant \sqrt{2 \varepsilon + \varepsilon ^ 2} + \varepsilon.
\end{equation}

Which is exactly the estimate on $\frac{|O'I_k|}{R},$ that we wanted to obtain:

\begin{equation}
	\frac{|O'I_k|}{R} \leqslant  \sqrt{2 \varepsilon + \varepsilon ^ 2} + \varepsilon.
	\label{eq_odashi}
\end{equation}

Now, combining inequalities (\ref{eq_oodash}) and (\ref{eq_odashi}), we get that if $k > K_1$ (and so $\varepsilon < \frac{1}{1000}$):

\begin{equation} 
\frac{|O_k I_k|}{R}
\leqslant \frac{ |O_k O'| + |O' I_k|}{R}  \leqslant 
 \arccos \left( 2 - \frac{1}{1-\varepsilon}\right) +
(\sqrt{2 \varepsilon + \varepsilon ^ 2} + \varepsilon).
\end{equation}

So 
\begin{equation}
	\forall \delta > 0 \enskip \exists K: \forall k > K \quad \frac{|I_k O_k|}{R_k} < \delta, 
\end{equation}
or in other words:
\begin{equation}
	\forall \delta > 0 \enskip \exists K: \forall k > K \quad |I_k O_k| < \delta |R_k|.
	\label{oi_upper}
\end{equation}

\textbf{2. Estimate on $|O_k O_{k+2}|$ from above.}
Now let's show that $O_{k+2}$ is close to $O_k.$
Denote $|O_k I_k|$ by $d.$ Let's instead of $R_k,$ again write $R$ for the radius of circumscribed sphere of $S_k.$ 
Consider such $K_2 > K_1,$ that $\forall k > K_2$ holds $\frac{d}{R} < \frac{1}{1000}.$

Consider further only these $k>K_2.$

We will reuse the names of the points from the previous chapter,
with only difference, that $\Gamma$ can be any full-dimension face of $S_{k+1},$ not only the one closest to $O_k$.
\begin{figure}[htp]
    \centering
    \includegraphics[width=14cm]{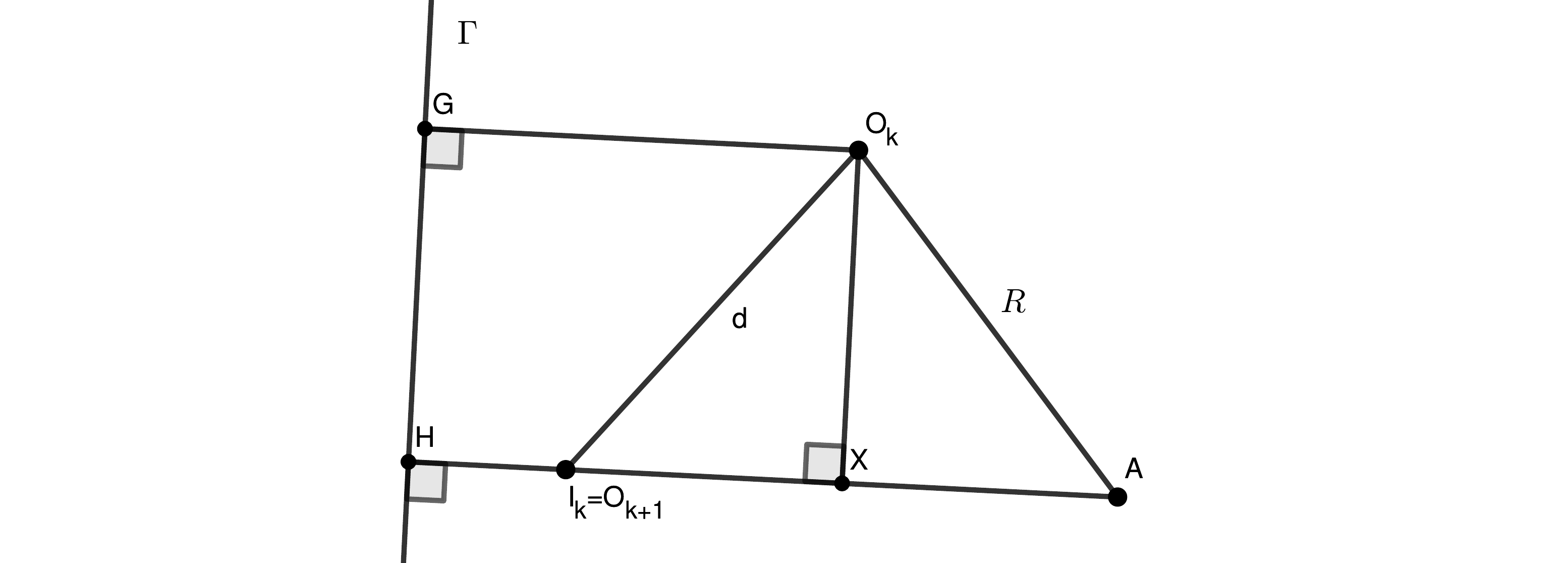}
    \caption{Estimate on $O_k O_{k+2}$ from above.}
    \label{fig:projection5}
\end{figure}

From (\ref{prod_property}) it follows:
\begin{equation}
	|AH| =|AI_k| + |I_k H| = |AI_k| + \frac{R^2}{|AI_k|}.
\end{equation}

From the triangle inequality for $\triangle I_k O_k A,$ we estimate $|AI_k|:$
\begin{equation}
	R-d \leqslant |AI_k| \leqslant R+d,
\end{equation}
so, as $f(x)=x+\frac{R^2}{x}$ is convex for $x>0,$ and attains maximal value on the interval at one of the intervals ends:

\begin{equation}
	|AH| \leqslant \max\left\{(R+d)+\frac{R^2}{R+d}, (R-d)+\frac{R^2}{R-d}\right\}.
	\label{AH_bound}
\end{equation}

Denote $\frac{d}{R}$ by $\delta.$ So,  (\ref{AH_bound}) is equivalent to:
\begin{equation}
	|AH| \leqslant R \max\left\{(1+\delta)+\frac{1}{1+\delta}, (1-\delta)+\frac{1}{1-\delta}\right\}.\end{equation}

Remember that as $k > K_2,$ we have $\delta < \frac{1}{1000},$ which leads to:
\begin{equation}
	\max\left\{1+\delta+\frac{1}{1+\delta}, 1-\delta+\frac{1}{1-\delta}\right\} \leqslant 2 + 2\delta^2.
\end{equation}

So 
\begin{equation}
	|AH| \leqslant R(2 + 2\delta^2).
	\label{eq_ah}
\end{equation}

On the other hand, since $|O_k X| \leqslant |O_k I_k| = d$ we get:
\begin{equation}
	|XA| =  \sqrt{R^2 - |O_k X|^2}
	\geqslant \sqrt{R^2 - d^2} = R \sqrt{1-\delta ^ 2}\geqslant R(1 - \delta^2).
	\label{eq_ax}
\end{equation}

Point $X$ lies inside the interval $[AH]$: since otherwise $d = |O_k O_{k+1}| \geqslant |AX| \geqslant R(1 - \delta^2).$
(Reader may notice that we have already shown in the previous chapter that $X$ lies on $[AH].$ We did, but only for one face $\Gamma$ of $S_{k+1}.$)
Now, combining inequalities (\ref{eq_ah}) and (\ref{eq_ax}) we get

\begin{equation}
	|O_kG| = |AH| - |AX| \leqslant R(2 + 2\delta^2 - (1-\delta^2)) = R(1+3\delta^2)
	=R+\frac{3d^2}{R}.
\end{equation}

In the Lemma \ref{lemma_container}, in (\ref{lower_bound_og}) we had a lower bound on $|O_k G|$, so together

\begin{equation}
R \leqslant |O_kG| \leqslant R + \frac{3d^2}{R}.
\label{eq_two_bounds}
\end{equation}

Starting from now let's again write $R_k,$ not $R.$

Let's look again at the whole simplex $S_{k+1}$ (we used to only look at one of its full-dimension faces, $\Gamma$).
Its radius of inscribed sphere equals $r_{k+1},$ and $I_{k+1}$ is its center of inscribed sphere.
In the Lemma \ref{lemma_container}, in (\ref{container_inequality}), we've shown that 

\begin{equation}
r_{k+1} \geqslant R_k.
\label{rkbig}
\end{equation}
Consider the smaller simplex $S'_{k+1},$ that is homothetic to $S_{k+1}$ with the coefficient $\frac{r_{k+1} - R_k}{r_{k+1}}$ (which is greater or equal than 0)
and center $I_{k+1}.$ Simplex $S'_{k+1}$ can be a single point. See Figure \ref{fig:smaller-simplex}.

\begin{figure}[htp]
    \centering
    \includegraphics[width=12cm]{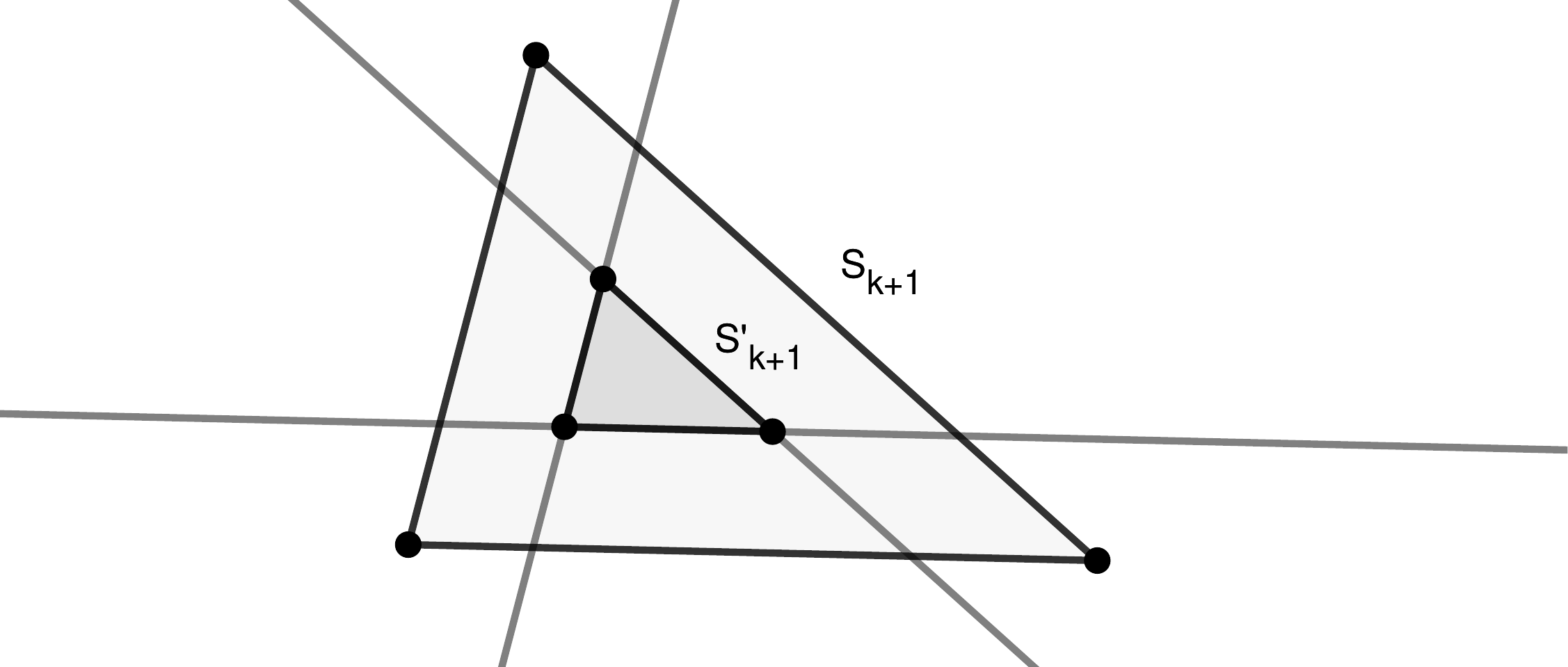}
    \caption{Simplices $S_{k+1}$ and $S'_{k+1}$.}
    \label{fig:smaller-simplex}
\end{figure}

The radius of the inscribed sphere of $S'_{k+1}$ is, by construction, $r_{k+1} -R_k$, which is the distance from
$I_{k+1}$ to any of the full-dimension faces of $S'_{k+1}.$
 So, the distance between the parallel full-dimension faces of $S_{k+1}$ and $S'_{k+1}$ equals
 $r_{k+1} - (r_{k+1} - R_k) = R_k.$

We now aim to show, that this tiny simplex $S'_{k+1}$ contains two points, mentioned in the beginning of the chapter: $O_k$ and $O_{k+2}.$ It is very simple with the point $O_{k+2},$ indeed, since $S_{k+2}= Root(S_{k+1}),$ from the Property \ref{prop_radii}, points $I_{k+1}$ and $O_{k+2}$ coincide:
\begin{equation}
	I_{k+1}=O_{k+2},
\end{equation}
and $I_{k+1}$ is the center of inscribed sphere of $S'_{k+1},$ so it lies inside $S'_{k+1}$.

Now fix some full-dimensional face of $S_{k+1},$ let's call it $\Gamma.$ It divides the space into two half-spaces,
let's denote one containing $S_{k+1}$ by $\Pi.$

Let the image of $\Gamma,$ after the homothety transforming $S_{k+1}$ into $S'_{k+1}$ be $\Gamma',$ and
the image of $\Pi$ be $\Pi'.$
Since $I_{k+1},$ the center of homothety lies in $\Pi,$ then $\Pi' \subset \Pi.$

From the Lemma \ref{lemma_container} the circumscribed sphere of $S_k$ lies inside $S_{k+1},$ and so does its center, $O_k.$
So $O_k \in \Pi.$
The distance from $O_k$ to $\Gamma$ is in the interval $[R_k;R_k+\frac{3d^2}{R_k}]$, see (\ref{eq_two_bounds}).
This means that since the distance between $\Gamma$ and $\Gamma'$ equals $R_k$,
that $O_k \in \Pi',$ and distance from $O_k$ to $\Gamma'$ is smaller or equal than $(R_k+\frac{3d^2}{R_k}) - R_k = \frac{3d^2}{R_k}.$

Repeating this argument for all full-dimension faces of $S_{k+1},$ we obtain that
$O_k$ lies inside $S'_{k+1},$ as $S'_{k+1}$ is the intersection of $n+1$ half-spaces in the family of $\Pi'.$

Now, as we've shown that $O_k, O_{k+2} \in S'_{k+1},$ let's prove that this simplex is indeed small. First, if $S'_{k+1}$ is a point, then $|O_k O_{k+2}|=0,$ and any upper bound holds. If $S'_{k+1}$ is not a point, let's continue. 

The distance from $O_k$ to all full-dimension faces of $S'_{k+1}$ is no larger than 
$\frac{3d^2}{R_k}.$
This means that the radius of the inscribed sphere of $S'_{k+1}$ is no larger than
$\frac{3d^2}{R_k}$
The simplices $S_{k+1}$ and $S'_{k+1}$ are homothetic, so if the radii of inscribed and circumscribed spheres of $S'_{k+1}$ are $r'_{k+1}$ and $R'_{k+1},$ then from (\ref{lower_bound_radius}):

\begin{equation}
	\frac{r'_{k+1}}{R'_{k+1}}=\frac{r_{k+1}}{R_{k+1}}\geqslant \rho(1-\varepsilon).
\end{equation}
So,
\begin{equation}
R'_{k+1} \leqslant \frac{r'_{k+1}}{\rho(1-\varepsilon)} \leqslant \frac{3d^2}{\rho (1-\varepsilon)R_k},
\end{equation}
which means that since both $O_k, O_{k+2} \in S'_{k+1}$ then
\begin{equation}
|O_k O_{k+2}| \leqslant 2R'_{k+1} \leqslant \frac{6d^2}{\rho(1-\varepsilon)R_k} =
\frac{6|O_k O_{k+1}|^2}{\rho(1-\varepsilon)R_k}, \text{ for } k > K_2.
\end{equation}

And since for $k > K_2$  holds $\varepsilon < \frac{1}{2},$ then:
\begin{equation}
	|O_k O_{k+2}| \leqslant \frac{12|O_k O_{k+1}|^2}{\rho R_k}, \text{ for } k>K_2.
	\label{oo_bound}
\end{equation}

\textbf{3. Sequence $\{O_{2k}, k \in \mathbb{N}\}$ is a Cauchy sequence.}
We have shown that points $O_{2k}$ and $O_{2k+2}$ are close to each other for sufficiently large $k.$ If we also can show, that $\sum\limits_{k=1}^{\infty} |O_{2k} O_{2k+2}| < \infty,$ this will lead to $|O_{2k}O_{2l}| \to 0, \enskip k,l \to \infty,$ meaning that $O_{2k}$ is a Cauchy sequence, and so, has a limit in $\mathbb{R}^n$

From Property (\ref{prop_radii}) and Corollary (\ref{corollary_limiting_radius}) we know 
\begin{equation}
R_{k+1} = R_k \cdot \frac{R_k}{r_k} \geqslant \frac{R_k}{\rho}.
\end{equation}

So, since we know from the Property \ref{rho_bound}, that $\rho \leqslant \frac{1}{n-1}$,
then $R_k$ grows fast, more precisely :

\begin{equation}
	R_k \geqslant R_1 \left(\frac{1}{\rho}\right)^{k-1}.
	\label{rk_grows}
\end{equation}

Let $c:=\frac{12}{\rho},$ and denote $d_k:=|O_k O_{k+1}|.$

Then for $k>K_2$ from (\ref{oo_bound}) we get:

\begin{equation}
	d_{k+1} = |O_{k+1}O_{k+2}| \leqslant |O_{k+1}O_{k}| +|O_kO_{k+2}| \leqslant d_k+ \frac{c d_k^2}{R_k}.
	\label{dk}
\end{equation}

Let $\varepsilon_2>0$ be a small real number, such that
\begin{equation}
\frac{(1 + c\varepsilon_2)^2}{ \frac{1}{\rho}} < \frac{1}{2}
\label{eps2def}
\end{equation}

From (\ref{oi_upper}) we have: for this $\varepsilon_2$ there exists such $K_3>K_2,$ that for all $k > K_3$ holds:
\begin{equation}
	\frac{d_k}{R_k} \leqslant \varepsilon_2,
	\label{dk2}
\end{equation}
so for $k>K_3$ the combination of (\ref{dk}) and (\ref{dk2}) gives:
\begin{equation}
d_{k+1}  \leqslant d_k+ \frac{c d_k^2}{R_k} \leqslant d_k + c\varepsilon_2 d_k = (1 + c\varepsilon_2) d_k,
\end{equation}

so for some large $C>0$ and for all $k$
\begin{equation}
	d_k \leqslant C(1 + c\varepsilon_2) ^ k.
	\label{dk_small}
\end{equation}

We can now continue the estimate on $|O_k O_{k+2}|,$ in (\ref{oo_bound}), for $k > K_2.$
First, equivalently to (\ref{oo_bound}):
\begin{equation}
	|O_k O_{k+2}| \leqslant \frac{c d_k^2}{R_k},
\end{equation}
then from (\ref{rk_grows}) for all $k$:
\begin{equation}
	\frac{c d_k^2}{R_k} \leqslant
	\frac{c d_k^2}{R_1\left( \frac{1}{\rho}\right)^{k-1}},
\end{equation}
then from (\ref{dk_small}), for all $k:$
\begin{equation}
	\frac{c d_k^2}{R_1\left( \frac{1}{\rho}\right)^{k-1}}
	\leqslant
	\frac{c\cdot C^2(1 + c\varepsilon_2)^{2k}}{R_1\left( \frac{1}{\rho}\right)^{k-1}} =
	\frac{c\cdot C^2}{\rho R_1} \left( \frac{(1 + c\varepsilon_2)^2}{ \frac{1}{\rho}} \right)^k.
\end{equation}

So, for $k>K_2$:
\begin{equation}
	|O_k O_{k+2}| \leqslant
	\frac{c\cdot C^2}{\rho R_1} \left( \frac{(1 + c\varepsilon_2)^2}{ \frac{1}{\rho}} \right)^k.
\end{equation}

And from (\ref{eps2def}), for $k>K_2$:
\begin{equation}
	|O_k O_{k+2}| \leqslant
	\frac{c\cdot C^2}{\rho R_1} \left( \frac{1}{2} \right)^k.
\end{equation}

The sum of the geometric progression with a ratio smaller than 1 is finite. So the sequence $\{O_{2k}\}$ is a Cauchy sequence, and so it has a limit in $\mathbb{R}^n.$

The same argument holds for $\{O_{2k+1}\},$ since it is the sequence of even order circumcenters, with the starting simplex $S_2$ instead of $S_1.$
\end{proof}

\bibliography{root} 
\bibliographystyle{ieeetr}

\end{document}